\documentclass[11pt,reqno]{amsart}

\usepackage{amsmath}
\usepackage{amsfonts}
\usepackage{amssymb}
\usepackage{verbatim}
\usepackage[left=2.6cm,right=2.6cm,top=3cm,bottom=3cm]{geometry}
\usepackage{color}
\usepackage{enumerate}

\usepackage{graphicx}

\makeatletter
\newcommand*{\rom}[1]{\expandafter\@slowromancap\romannumeral #1@}
\makeatother

\newcommand{\calS}{{\mathcal S}}

\newcommand{\mP}{{\mathbb P}}

\newcommand{\N}{\mathbb N}

\newcommand{\Gb}{\mathbb G}
\newcommand{\Hb}{\mathbb H}

\newtheorem{theorem}{Theorem}
\newtheorem{lemma}{Lemma}
\newtheorem{rem}{Remark}

\newtheorem{example}{Example}
\newtheorem{definition}{Definition}
\newtheorem{corollary}{Corollary}

\title[Ranking graphs through   hitting times]{Ranking graphs through   hitting times of Markov chains
}

\author{Emilio De Santis}
\email{desantis@mat.uniroma1.it}

\address{University of Rome La Sapienza, Department of Mathematics. Piazzale Aldo Moro, 5, 00185, Rome, Italy}

\begin{document}

\begin{abstract}
 In the present paper we show that for any  given digraph $\Gb =([n], \vec{E})$, i.e. an oriented graph  without self-loops and 2-cycles, one can construct a 1-dependent Markov chain and $n$ identically distributed hitting times $T_1, \ldots , T_n $ on this chain such that the probability of the  event
$T_i > T_j $, for any $i, j = 1, \ldots n$, is larger than $\frac{1}{2}$ if and only if $(i,j)\in \vec{E}$.
This result is related to various \emph{paradoxes} in  probability theory, concerning in particular non-transitive dice.
\end{abstract}

\maketitle
\section{Introduction}\label{introb}


Let us consider a collection of  random variables $\mathcal{Y}_n = \{Y_1, \dots, Y_n\}$, defined on a
same probability space and satisfying the no-tie condition,  i.e. any two of them are equal with probability zero. 
Let $ [n]:=\{1, 2,\dots ,n\}$.
For $ i \neq j \in [n]$, one
says that $Y_i$ is \emph{less than} (resp. is \emph{equivalent to}) $Y_j$ in the \emph{stochastic precedence} sense if
$$
\mP (Y_i < Y_j) > \frac{1}{2}, \qquad \left ( \text{resp. } \mP (Y_i < Y_j ) = \frac{1}{2} \right ).
$$
This notion 
is natural in many
applications, see e.g. \cite{AS2000, blyth72, DSS1} and 
references therein. From a theoretical point of view, it gives rise to a series of apparent paradoxes which are 
caused by the non-transitivity of the stochastic precedence comparison, 
see  \cite{DMS19, Savage, trybula1969}.

In order to present some results it is convenient to introduce the definition of   \emph{ranking graph}  $ \Gb (   \mathcal{Y}_n    )   = ([n], \vec{E}   (   \mathcal{Y}_n    )   )$ associated with the  set of random variables $\mathcal{Y}_n = \{Y_1, \dots, Y_n\}$.  For any $i,j \in [n]$, the pair  $(i,j)$ is an arrow of 
$\vec{E}   (   \mathcal{Y}_n    )  $ if and only if $Y_j $ is less than  $Y_i $ in the stochastic precedence sense.  

Although expressed in a different form, it is known that for any digraph $\Hb =([n], \vec E)$ there exists a set of $n $ random variables $ \mathcal{Y}_n $ such that the ranking graph $\Gb (   \mathcal{Y}_n    ) $ 
coincides with $\Hb  $, i.e. any set of stochastic precedence relations can be achieved (see \cite{McGarvey1953, Saari}).
This result was initially obtained in the field of voting theory by McGarvey, see \cite{McGarvey1953}, but it  can  be rewritten by using random  variables and the stochastic precedence comparison between them  (see \cite{blyth72, DS20, Saari}). We also mention \cite{S1959},  \cite{ErdosMoser}, \cite{AlonPara}  and \cite{Dominating2006} which improve McGarvey's result by providing precise upper and lower bounds on the number of voters needed to achieve all possible ranking graphs. 


In this paper we take a different point of view and we give a positive answer to the following problem: given a digraph $ \Hb   =([n], \vec E) $, can we find a stationary Markov chain with $n$ \emph{identically distributed} hitting times $\mathcal{T}_n = \{T_1, \ldots , T_n \}$ such that $\Gb (   \mathcal{T}_n    ) =\Hb  $?
Our interest in this problem is twofold. 
\begin{itemize}
\item[1)] The problem loses its combinatorial structure and becomes entirely probabilistic. 
\item[2)] There are many applications where it is preferable to use hitting times rather than arbitrary random variables.
\end{itemize}

In various applications the type of random variables to use is constrained by the nature of the 
problem.
For instance, in many competitive games the winner is the player who achieves his goal or target before the others. In this situation 
one is forced to work with hitting times. As an example of an application in this field, in the final part of the paper 
 we define and analyze a Penney-type game in which  these random variables arise quite naturally.

From a theoretical  point of  view, we emphasize that the structure of the hitting times is very particular and many properties on their distributions are known a priori, see e.g. \cite{DSS5},
\cite{MaSc} and references therein. Moreover, the stationary Markov chains used in our construction have the particular property that the square of the transition matrix has all of its entries equal. Hence, they are \emph{1-dependent uniform}, in the sense that they are $1$-dependent \cite{1-dependent, 1-bis}, with uniform  invariant distribution.


The plan of the paper is as follows.

In the next Section \ref{uniform} we give some basic notation and definitions, and preliminary results. In
Section \ref{realization}, we present our  main result through an explicit construction of  1-dependent uniform Markov chains.   In Section \ref{comb}, we present a Penney-type games and we develop a qualitative analysis of it. In particular, we identify a threshold value beyond which it is possible to construct a Penney-type game unfavorable to the starting  player whereas this is impossible below the same threshold.

\section{1-dependent uniform Markov chains, patterns and identically distributed hitting times} \label{uniform}

First of all, recall that
for a collection of r.v. $\mathcal{Y}_n= \{Y_1, \ldots , Y_n\}$, we
 defined the \emph{ranking graph}  $ \Gb (   \mathcal{Y}_n    )   = ([n], \vec{E}   (   \mathcal{Y}_n    )   )$ in the  following way.  For any $i,j \in [n]$,
\begin{equation} \label{realizza}
(i,j ) \in \vec E     (   \mathcal{Y}_n    )   \iff  \mP (Y_j < Y_i)   > \frac{1}{2}  .
\end{equation}
 It is clear  that $  \Gb (   \mathcal{Y}_n    )    $ does not have loops or 2-cycles, thus  it is a digraph.

We also  recall the definition of 1-dependent uniform chain. 
\begin{definition}\label{1uc}
A  Markov chain  is a \emph{1-dependent uniform chain} if it is stationary and the square of its transition matrix has all equal entries.  The set of  1-dependent uniform chains is denoted by  $\mathcal{M}_1$.
%
%
\end{definition}
We notice that a 1-dependent uniform chain $\mathbf{X} =(X_m :m \in \N_0)$ has the property that, 
 for any sequence of increasing indices $m_1, m_2, \ldots $ with $m_{\ell+1} - m_{\ell } \geq 2$ (for $\ell \in \N $),   the  
random variables  
$(X_{m_\ell} : \ell \in  \N )$ are  i.i.d.  with $X_{m_\ell}$ distributed uniformly on the state space. 
 In particular, 
the initial distribution 
is invariant and it coincides with the uniform distribution. For further properties on 1-dependent chains see \cite{1-dependent}, \cite{1-bis},  where such chains are studied in detail and characterized.

\subsection{Construction of 1-dependent uniform Markov chains.}

Let $k, N \in \N $ with $N\geq 2$,  we  consider a sequence of i.i.d. random (column) vectors 
$\mathcal{U}_m =(U_{m,1}, \ldots , U_{m,k})^T$, with $U_{m, i}$ i.i.d. uniformly distributed on $[N]$, for $i \in [k]$ and  $m \in \N_0$.
%
Next define the $k \times 2 $ matrix 
\begin{equation}\label{duecolonne}
X^{(N,k)}_m = [\mathcal{U}_m, \mathcal{U}_{m+1}] , \text{ }m \in \N_0.
\end{equation}
Notice that the collection of random variables
$\mathbf{X}^{(N,k)}=   (X^{(N,k)}_m: m \in \N_0)  $ forms a $1$-dependent uniform chain, in particular it starts
with the uniform distribution on the state space  of all the matrices $k \times 2 $ with elements belonging to $[N]$
 (see Example \ref{espattern}). We notice that, for any $k\in \N$ and $N \geq 2$ the Markov chain $\mathbf{X}^{(N,k)}$ is not reversible,
instead  $\mP ( X^{(N,k)}_0 = [\mathbf{1}, \mathbf{1}] ,  X^{(N,k)}_1 = [\mathbf{1}, \mathbf{2}] ) = N^{-3k} $ but $\mP ( X^{(N,k)}_0 = [\mathbf{1}, \mathbf{2}] ,  X^{(N,k)}_1 = [\mathbf{1}, \mathbf{1}] ) =0 $, where $\mathbf{1}$ and $ \mathbf{2}$ are column vectors with all the entries equal to $1$ and $2$, respectively.

\subsection{Patterns}
For  $M , k \in  \N $,  a \emph{pattern} $Q=  (q_{i,j} \in  [M] \cup \{0 \}   : i \in [k], j \in [2])   $ is a  $k \times 2 $ matrix, with the property that 
$$
q_{i, 1} \in [M ],  \text{ for } i \in [k], \text{ and } 
\sum_{i=1}^k \mathbf{1}_{\{  q_{i,2} \neq 0     \}} =1.
$$
  For a pattern $Q$ we define the \emph{index of  hump} as 
\begin{equation} \label{indice}
h(Q) = j \text { if } q_{j, 2} >0. 
\end{equation}
The collection of all the 
 $k \times 2 $  patterns  with entries
 in $ [M] \cup \{ 0 \}$ is denoted by $\mathcal{P}_{M, k}$.
In particular,  any  pattern in   $    \mathcal{P}_{M, k} $ has a
 number of entries different from zero which is equal to  $k+1$.

\subsection{Hitting time of a pattern}

For any $h \in [k]$, we define the projection $\Pi_h$ on the set of 
$k \times 2 $ matrices in the following way:  
if $ S=(s_{i,j}    : i \in [k], j \in [2])$ , then 
\begin{equation} \label{proietti}
\Pi_h (S) =   
\begin{bmatrix}
       s_{1,1} & 0  \\
\vdots & \vdots \\
s_{h-1,1} & 0 \\
       s_{h,1} & s_{h,2} \\
s_{h+1,1} & 0 \\
       \vdots & \vdots \\
s_{k,1} & 0
     \end{bmatrix}   .
\end{equation}


Now let  $\mathbf{X}^{(N,k)}      =(X^{(N,k)} _m:m \in \N_0)  $ be 
as in \eqref{duecolonne}
and  let   $R \in \mathcal{P}_{M,k}$ be a target pattern, with $2 \leq M \leq N$. We define the hitting time of $R$ as
\begin{equation}\label{stopping}
  T_R =   \inf
   \{ m \in \N_0 :  \Pi_{h(R)} ( X^{(N,k)} _m ) =  R   \} , 
\end{equation}
In the following when  $ \Pi_{h (R)} ( X^{(N,k)} _m ) =  R$ holds we will write
$ X^{(N,k)} _m \triangleright R$. 

The hitting time  $T_R $ can be interpreted as the first time in which the pattern $R $ occurs  in the random sequence   $(X^{(N,k)} _m)_{ m\in \N_0}$. From the finiteness of state space $[N]^k$ and the irreducibility of the Markov chain $\mathbf{X}^{(N, k)}$ it follows that $T_R$ is finite almost surely. 

\subsection{Overlap}
For two different patterns
\begin{equation}\label{duepat}
R =( r_{i,j}  \in [M]    : i\in [k], j   \in[2]  )    ,  S=( s_{i,j} \in [M]    : i\in [k], j   \in[2]  )  \in \mathcal{P}_{M,k} ,
\end{equation}
 we define the \emph{overlap}
$ O(   R  , S   )\in  \{0,1\}^2 $ in the following way 
\begin{equation}\label{overlapvec1}
  O_1 (R ,S )=
\left \{
  \begin{array}{ll}
    1, & \text{ if }   r_{h  , 1} =
  s_{h, 1 } , \text{ }h \in [k] \text{ and }  \delta_{h(R) , h(S)} (r_{h(R), 2}- s_{h(S), 2})=0
 ; \\
    0, & \text{ otherwise. } \\
  \end{array}
\right .
\end{equation}
\begin{equation}\label{overlapvec2}
  O_2 (R ,S )=
\left \{
  \begin{array}{ll}
    1, & \text{ if }r_{h(R), 2} = s_{h(R), 1}  
 ; \\
    0, & \text{ otherwise. } \\
  \end{array}
\right .
\end{equation}
An analogous definition for strings has given in \cite{GO1981}.
We  notice that 
$T_R $ and $ T_S $ are different w.p. 1 if and only if 
$O_1(R ,S )=0 $. The meaning of
$O_2(R ,S ) $ will be clarified by Lemmas \ref{richiesto}-\ref{lemmastima} and Theorem \ref{thgrfo}  in which the relevance of $O_2(R ,S ) $ for the computation of $\mP(T_R < T_S)$ becomes evident.

 The overlap $O(R, S)$ is in general different from $ O(S, R)   $. Moreover it makes sense to consider the overlap of a pattern with itself, 
  in which  case the first component is always
  equal to  $1$.   
 For  patterns $ R_1, \ldots , R_n \in \mathcal{P}_{M,k}$, one has
  \begin{equation}\label{sipuo}
 \text{no-tie property of }  T_{R_1}, \ldots , T_{R_n}
 \iff   \text{for  distinct } i,j \in [n] , \text{ } O_1(R_i, R_j) =0 .
  \end{equation}
  In this case, for sake of simplicity,  we say that the collection of patterns $R_1, \ldots , R_n $ is \emph{no-tie}.


To familiarize with the notions and definitions we present the following example

\begin{example}  \label{espattern}
Let
$$
R=
     \begin{bmatrix}
       1 & 0  \\
       2 & 1 \\
       2 & 0
     \end{bmatrix}
   , \,\,\,
 S=
     \begin{bmatrix}
       1 & 0  \\
       1 & 0 \\
       2 & 1
     \end{bmatrix} ,
$$
be two patterns in $\mathcal{P}_{2, 3}$, with index of hump $i (R) =2 $ and $i (S)= 3$.
The overlap $O(R, S) =(0,1 )$ and $O( S, R) =(0, 0)$. Thus the collection of patterns $\{ R,  S \}$ is no-tie.

\noindent
Suppose that
$$\mathbf{X}^{(2,3)}=
     \begin{array}{ccccccc}
       1 & 1 & 1 & 1  &  1 & 2 &  \cdots  \\
       2 & 2 & 1 & 1  & 2 & 1 & \cdots \\
       2 & 1 &  2 & 1 & 2 & 1 & \cdots  \\
     \end{array}
$$
Then  $T_S =2 $ being
$$
S \neq \Pi_3 \begin{bmatrix}
         1 & 1   \\
        2 & 2   \\
        2 & 1
     \end{bmatrix} =
\begin{bmatrix}
         1 & 0   \\
        2 & 0   \\
        2 & 1
     \end{bmatrix}
 , \quad
S \neq \Pi_3 \begin{bmatrix}
         1 & 1   \\
        2 & 1   \\
        1 & 2
     \end{bmatrix} =
 \begin{bmatrix}
         1 & 0   \\
        2 & 0   \\
        1 & 2
     \end{bmatrix} 
, \quad
S = \Pi_3 \begin{bmatrix}
         1 & 1   \\
        1 & 1   \\
        2 & 1
     \end{bmatrix}  
= \begin{bmatrix}
         1 & 0   \\
        1 & 0   \\
        2 & 1
     \end{bmatrix}
  .
$$
Similarly,
$T_R = 4$.
\end{example}

\section{Ranking graphs through identically distributed hitting times } \label{realization}

Let  $R$ be a pattern with $O(R, R) =(1,0)$.
Following the proof of Theorem 2.1 in  \cite{BT82} we  recursively compute the discrete distribution of  $T_{R} $.
\begin{lemma} \label{tt11nuovo}
Let $k,M,N$ be integers such that $N \geq M \geq 2$ and $k\geq 1$. Let $R \in \mathcal{P}_{M,k}$ with $O(R, R) =(1,0)$ and  consider the 1-dependent uniform chain $\mathbf{X}^{(N , k)} $. Define
$w  (t):= \mathbb{P} ( T_{R} =t)$, then 
the probabilities $(w(t) : t \in \N_0 )$ are recursively determined
 from
\begin{equation} \label{iterate}
w(t) =  N^{-k-1} - N^{-k-1}
 \sum_{s = 0}^{t-2} w (s) .
\end{equation}
\end{lemma}
\begin{proof}
Consider the event $\{  X^{(N , k)}_t    \triangleright R   \}$,  it holds true with
 probability $N^{-k-1}$. Moreover, $\{  X^{(N , k)}_t    \triangleright R   \}$ can be written as the union of the following three disjoint events: 
\begin{itemize}
\item[(i)]  $  \{ T_{R} = t   \}$;
\item[(ii)]  $ \{  X^{(N , k)}_t   \triangleright R   \} \cap  \{ T_{R } = s  \}$, for $s < t-1 $;
\item[(iii)] $ \{  X^{(N , k)}_t   \triangleright R    \} \cap  \{ T_{R } = t-1  \}$.
\end{itemize}
The probability of $  \{ T_{R } = t   \}$ is by definition $w (t)$.

For $s < t-1 $, as a consequence of 1-dependence, one has
$$
\mathbb{P}   (    \{  X^{(N , k)}_t   \triangleright R    \} \cap  \{ T_{R } = s  \}  ) =
w(s )\mathbb{P}   (    X^{(N , k)}_t   \triangleright R     )  = N^{-k-1}   w (s ) .
$$
 The event in (iii) has probability zero because,
by hypothesis  $O_2(R, R) =0$.
Therefore
$$
N^{-k - 1} =   w (t ) +  N^{-k-1}
 \sum_{s = 0}^{t-2} w (s) ,
$$
which corresponds to \eqref{iterate}.
\end{proof}
\begin{rem}\label{remare}
Let  $ R_1, \ldots , R_n \in  \mathcal{P}_{M,k}$ such that $O(R_i,R_i) =(1,0)$, for each $i \in [n]$. Let $N \geq M$ and consider the 1-dependent uniform chain
$\mathbf{X}^{(N , k)} $. Then, by Lemma \ref{tt11nuovo}, the hitting times $T_{R_1}, \ldots ,   T_{R_n}  $ are identically distributed because the distribution of any $T_{R_i}$
is given by \eqref{iterate}.
\end{rem}

Given a digraph $ \bar \Gb=([n ],  \vec{E} )$, we now construct a collection of associated patterns $\{ R_u \in \mathcal{P}_{n+1,n+1}:u \in [n]\}$.
For $\ell \in [n]$,  pattern $R_\ell =(r^{(\ell)}_{i,j}  : i \in [n+1] , j\in [2]         ) $ is constructed in the following way:
\begin{itemize}
\item[\emph{1}.] $r^{(\ell)}_{1,1} =r^{(\ell)}_{\ell +1 ,2 } =\ell  $;
  \item[\emph{2}.] for any $j \in [n +1] \setminus \{ \ell+1 \}$,  $ r^{(\ell)}_{j, 2} =0$;
  \item[\emph{3}.] for any $j \in [n ]     $,
  \begin{equation} \label{over}
   r^{(\ell )}_{j+1,1} = \left \{
                            \begin{array}{ll}
                              j, & \text{ if  $(\ell,j   )\in \vec E $;} \\
                             n+1, & \text{ otherwise. } \\
                            \end{array}
                          \right .
   \end{equation}
\end{itemize}
We will say that the patterns $R_1, \dots , R_n $ are \emph{generated} by
the graph $\bar \Gb =([n], \vec E)$.

First notice that $h (R_\ell ) =  \ell +1 $, for $\ell \in [n]$. 
In order to explain our construction of the patterns we observe that when $(i, j) \in \vec{E}$ then
$O_2(R_i, R_j) = 0$ and $O_2(R_j, R_i)= 1$ (see Lemma \ref{tt11} below), and this will cause $  \mP (T_{R_j}    < T_{R_i}  ) > \frac{1}{2}   $ through the phenomenon of clustering (see Lemma \ref{richiesto}-\ref{lemmastima} and Theorem \ref{thgrfo}). This phenomenon is analogous to what happens to the appearance of strings in a sequence of letters randomly drawn (see e.g. \cite{chen1979, GO1981}). To illustrate the notation, we present the following example.

\begin{example}  \label{escostru}
Let us consider the graph $ \bar \Gb=([3],  \vec{E} )$ with $ \vec{E} = \{ (1,3), (3,2), (2,1) \}$. Then the 
patterns $R_1, R_2 , R_3$ generated by $ \bar \Gb   $ are 
$$
R_1=
     \begin{bmatrix}
       1 & 0  \\
       4 & 1 \\
       4 & 0  \\
       3 & 0
     \end{bmatrix}
   , \qquad
 R_2=
     \begin{bmatrix}
       2 & 0  \\
       1 & 0 \\
       4 & 2  \\
       4 & 0
     \end{bmatrix} ,
\qquad
 R_3=
     \begin{bmatrix}
       3 & 0  \\
       4 & 0 \\
       2 & 0  \\
       4 & 3
     \end{bmatrix}.
$$
The present example will be continued at the end of the section where the patterns will be employed in the definition of three hitting times that show non-transitivity for the stochastic precedence.
\end{example}

We now consider the properties of the overlap for patterns $R_1, \dots , R_n $  generated by a digraph $\bar \Gb =([n], \vec E)$.

\begin{lemma} \label{tt11}
Let $n \geq 2$ and  $\bar \Gb =([n], \vec E)$ be a digraph. The overlaps of the  patterns  $R _1 \ldots , R_n $  generated by  $\bar \Gb =([n], \vec E)$
are
$$
O(R_i, R_j) = \delta_{i,j} \cdot (1,0) + ( 1-\delta_{i,j}) \cdot \left [ 
\mathbf{1}_{\{ (i,j) \in\vec E\}}    \cdot        (0,1)  
+      \mathbf{1}_{\{ (i,j) \not \in\vec E\}}    \cdot   (0,0)  \right  ]   ,
$$
for any $i,j \in [n]$.
\end{lemma}
\begin{proof}
Case $i=j$, then $O_1(R_i, R_i) =1$. The second component $O_2(R_i, R_i) =0$ since $r^{(i)}_{ i+1, 1} =n+1 \neq i = r^{(i)}_{ i+1, 2}$ (see \eqref{over}).

Case $i \neq j $. $O_1(R_j, R_i) =0$ since $r^{(i)}_{1, 1} =i \neq j = r^{(j)}_{1, 1}$.

If  $(i,j)  \in \vec E$ then $O_2(R_j, R_i) =1$. Indeed, the condition in \eqref{overlapvec2}  
$
r^{(j)}_{h (R^{(j)})  , 2 } = r^{(i)}_{h (R^{(j)})  , 1 }
$ 
 holds true 
since $   r^{(j)}_{j+1, 2}   =    r^{(i)}_{j+1, 1}  = j$.

If  $(i,j)  \not \in \vec E$ then $O_2(R_j, R_i) =0$. Indeed, the condition   
$
r^{(j)}_{h (R^{(j)})  , 2 } = r^{(i)}_{h (R^{(j)})  , 1 }
$ 
 is not true 
since $   r^{(j)}_{j+1, 2}   =   j $ and $ r^{(i)}_{j+1, 1}  = n+1$.

\end{proof}

 Let us consider no-tie collection of patterns $\{ R_1, \ldots , R_\ell\}  $ belonging to $  \mathcal{P}_{M, k}$  and the corresponding hitting times $ \mathcal{T}_\ell =\{ T_{R_1}  , \ldots , T_{R_\ell}   \}$
of  $\mathbf{X}^{(N,k)}$, with $N \geq M$. We are interested to upper and lower bound
\begin{equation} \label{prob}
   p_i   (  \mathcal{T}_\ell )  : =\mP \left  (  \bigcap_{j \in [\ell ] } \{  T_{R_i}  \leq   T_{R_j}  \}         \right )   \text{ for } i \in [\ell ].
\end{equation}
First notice that  by the no-tie property one has
$\sum_{i \in [\ell ] } p_i   (  \mathcal{T}_\ell  )   =1$.

We also define the  sequence of stopping times $(Z_h: h \in \N_0 )$, as
\begin{equation}\label{Z1}
  Z_0= \inf \{ m \geq 0 :  X^{(N,k)}_m\triangleright R_i \text{ for some } i \in [\ell] \}
\end{equation}
and recursively, let
\begin{equation}\label{Zh}
  Z_{h+1}= \inf \{ m\geq   Z_{h} + 2 :  X^{(N,k)}_m\triangleright R_i \text{ for some } i \in [\ell]  \}.
\end{equation}
It is immediate to notice that every hitting time $ Z_h$ is finite almost surely. We present this simple lemma without a proof.

\begin{lemma}\label{richiesto}
For any $s\in \N$, one has
\begin{equation} \label{nonno}
\mathbb{P} (  X^{(N,k)}_{Z_s  +1   }  \triangleright R_i  |      X^{(N,k)}_{Z_s     }  \triangleright R_j    ) = \frac{O_2(R_j , R_i)}{N^k} .
\end{equation}
\end{lemma}

The probabilities $( p_i   (  \mathcal{T}_\ell  )  : i \in [\ell ])$  could be explicitly calculated through a linear system
but for our purposes it will be more useful to have good upper and lower bounds.

\begin{lemma}  \label{lemmastima}
Let $N \geq M \geq 2 $ and $k , \ell \geq 2 $. 
Let $R_1, \ldots , R_\ell \in \mathcal{P}_{M, k}$ be a collection of no-tie patterns
 and let $O_2 (R_i, R_i) =0$, for $i \in [\ell ]$.
Let us take the uniform Markov chain $\mathbf{X}^{(N,k)}$ and the hitting times $    \mathcal{T}_\ell =\{ T_{R_i} : i \in [\ell ]     \} $.
Then $    \mathcal{T}_\ell  $  are identically distributed.
Moreover, for $i \in [\ell ]$,
\begin{equation} \label{probfine}
   p_i   (  \mathcal{T}_\ell )   =\frac{v_i}{\sum_{j \in [\ell]}  v_j   } ,
\end{equation}
where
\begin{equation} \label{vistima}
1- \frac{1}{N^k}  \sum_{j \in [\ell] }    O_2(R_j, R_i)  \leq v_i \leq 1 - \frac{1}{N^k}  \left (1- \frac{\ell -1 }{N^k} \right )
 \sum_{j \in [\ell] }  O_2(R_j, R_i )   . 
\end{equation}
\end{lemma}
\begin{proof} 
 In the proof, we will write $X_m $ for $X_m^{(N,k)}$. 
 Lemma \ref{tt11nuovo}, Remark \ref{remare} and
 $O_2 (R_i, R_i)=0   $, for each $i$,  imply  that $    \mathcal{T}_\ell$  are identically distributed.


For any $i \in [\ell ]$, by definition
$$
p_i( \mathcal{T}_\ell) = \mP (T_{R_i}    = Z_0  ) .
$$
From the fact that the sequence of random variables $(X_m) _{m \in \N_0 }$ is 1-dependent also
\begin{equation} \label{MP}
p_i( \mathcal{T}_\ell) = \mP (X_{Z_h } \triangleright R_i  ) ,
\end{equation}
for any $h \in \N$. Moreover the times $(Z_s : s \in \N_0)$ are renewal times i.e.
$$
\mP (X_{Z_h } \triangleright R_i , Z_h -Z_{h-1} = s|   X_{Z_{h -1}} =j   , Z_{h -1}=t        ) , \text{ for } h, s \in \N , \text{ and } i \in [\ell],
$$
does not depend on $h, t\in \N$ and $ j \in [n]$. Again it is a consequence of the 1-dependent structure.

We define, for any $i \in  [\ell]$, the sets of random times
\begin{equation}\label{SETVu}
 \mathcal{V}_{i, t} := \{ m <t : m=Z_s \text{ for some } s ,  X_{m} \triangleright R_i \}, \quad
 \mathcal{N}_{i, t} := \{ s<t: X_s \triangleright R_i   \} ,
\end{equation}
where $t \in \N \cup \{+ \infty \}$. The  cardinalities are
\begin{equation}\label{Vu}
 V_{i, t} := |   \mathcal{V}_{i, t}       |= \sum_{s=0}^{\infty} \mathbf{1}_{ \{ Z_s \leq t-1 \} }  \mathbf{1}_{  \{ X_{Z_s     }  \triangleright R_i  \} }, \quad
  N_{i, t} :=    |   \mathcal{N}_{i, t}       |=       \sum_{s=0}^{t-1}  \mathbf{1}_{  \{  X_s \triangleright R_i   \}  }.
\end{equation}
By \eqref{MP} and by the ergodic theorem for renewal process, one has
\begin{equation}\label{fra1bb}
  \lim_{t \to \infty} \frac{V_{i,t}}{\sum_{j \in [\ell ]}   V_{j,t}     } = p_i    (   \mathcal{T}_\ell   )         \,\,\, a.s.  , \quad \lim_{t \to \infty} \frac{N_{i,t}}{t} =\frac{1}{N^{k+1}}  \,\,\, a.s.
\end{equation}
We define the quantities $ (v_i>0 :i \in [\ell] )$ as
\begin{equation} \label{leobb}
v_i := \lim_{t \to \infty}\frac{V_{i,t}   N^{k+1}  }{t}  \,\,\, a.s.,
\end{equation}
 by
 hypothesis $N \geq M $   one has $v_i \leq 1$, for each $i \in [\ell ]$.
The  equalities in \eqref{fra1bb} and the previous definition give \eqref{probfine}.

We notice that
\begin{equation} \label{spiegare}
N_{i,t} = V_{i,t} + \sum_{j \in [\ell]: j\neq  i} \, \sum_{s=0}^{\infty} \mathbf{1}_{ \{ Z_s \leq t -2\} }  \mathbf{1}_{  \{ X_{Z_s     }  \triangleright R_j \}  }
  \mathbf{1}_{  \{ X_{Z_s  +1   }  \triangleright R_i  \}  }    ,
\end{equation}
indeed if $ X_{Z_s  +1}  \triangleright R_i $, for some $s $, the time $(Z_s+1 )$ belongs to  $\mathcal{N}_{i, \infty}$ but it is  not in $\mathcal{V}_{i, \infty} $.
Let us multiply  by $N^{k+1}  /t$ the previous formula and  take the limit for
$t \to \infty $, then, by the ergodic theorem, by  \eqref{fra1bb} and \eqref{leobb}, one obtains
\begin{equation} \label{nuovaF}
1 =  v_i +        \lim_{t \to \infty}
\frac{N^{k+1}}{t}       \sum_{j \in [\ell ]: j\neq  i} \,\sum_{s=0}^{\infty} \mathbf{1}_{ \{ Z_s \leq t -2\} }  \mathbf{1}_{  \{ X_{Z_s     }  \triangleright R_j \}  }
  \mathbf{1}_{  \{ X_{Z_s  +1   }  \triangleright R_i  \}  }        \qquad a.s.    
\end{equation}
By the ergodic theorem and Lemma \ref{richiesto} one has
\begin{equation}\label{ssh}
1= v_i + \sum_{j \in [\ell ]: j\neq  i}     \frac{O_2(R_j , R_i)}{N^k}          v_j .
\end{equation}
Thus,
\begin{equation}\label{ssh2}
1\leq   v_i  +  \sum_{j \in [\ell ]: j\neq i} \,\frac{O_2 (R_j , R_i)}{N^k} =
        v_i  +  \sum_{j \in [\ell ]} \,\frac{O_2 (R_j , R_i)}{N^k} .
\end{equation}
The inequality \eqref{ssh2} corresponds to the first  inequality in \eqref{vistima}. In particular, $ v_i \geq 1 - (\ell-1 ) /N^k$, for any $i \in [\ell ]$. Thus, for any fixed $i \in [\ell ]$
\begin{equation} \label{purequesta}
  \lim_{t \to \infty}
\frac{N^{k+1}}{t}      \sum_{s=0}^{\infty} \mathbf{1}_{ \{ Z_s \leq t -2\} }  \mathbf{1}_{  \{ X_{Z_s     }  \triangleright R_i \}  }
  \geq 1    - (\ell -1 ) /N^k    \qquad a.s.   
\end{equation}
Now, by \eqref{spiegare}-\eqref{purequesta}, Lemma \ref{richiesto} and ergodicity one has
$$
v_i \leq 1 - \frac{1}{N^k}  \left (1- \frac{  \ell  -1 }{N^k}  \right )  \sum_{j \in [\ell]  }    O_2(R_j, R_i ) .
$$
This end the proof.
\end{proof}

We are now ready to present the following result on the construction of any digraph through the ranking graphs of 
1-dependent uniform chains and identically distributed hitting times.
\begin{theorem} \label{thgrfo}
For any digraph $ \bar \Gb=([n],  \vec{E} )$ there exists  
$\mathbf{X}=(X_m: m \in \N_0) \in \mathcal{M}_1$ and a
collection of identically distributed hitting times $\mathcal{T}_n  =\left\{T_1,T_2,\ldots,T_n\right\}$ on $\mathbf{X}$ such that
$\Gb ( \mathcal{T}_n  )    = \bar \Gb$.
\end{theorem}
\begin{proof}
Let us consider the 1-dependent uniform  chain $\mathbf{X}^{(N, n+1)}   $ with $N \geq  n+1$.
By Lemma \ref{tt11},
the patterns  $R_1, \ldots , R_n $  generated by $\bar \Gb$ have the no-tie property. Furthermore, by Lemmas \ref{tt11nuovo}-\ref{tt11} the hitting times $\{T_{R_1}, \ldots, T_{R_n}\}$ are identically distrributed. 

For  distinct indices $i,j$, one has 
$$
p_a(\{ T_{R_i}, T_{R_j}\}) = \mP (T_{R_a}    = Z_1   ) ,
$$
where $a \in \{ i,j\}$.

In the case that $ (i,j)$ and $(j,i) $ are not arrows of  the digraph $ \bar{ \mathbb{G }  } $
then, by Lemma \ref{tt11}, $O_2(R_i, R_i) = O_2(R_j, R_j) =O_2 (R_i, R_j) =O_2(R_j, R_i) =0 $. Thus, by \eqref{spiegare} of Lemma \ref{lemmastima} follows that $N_{i,t} = V_{i,t}$ and $ N_{j,t} = V_{j,t}$. Hence by \eqref{fra1bb} and \eqref{leobb} one has
$$
p_i   (   \{ T_{R_i}, T_{R_j}\}   ) = p_j    (   \{ T_{R_i}, T_{R_j}\}   ) = \frac{1}{2}.
$$
Hence,  $ (i,j)  $ and $(j,i)$  does not belong to  $\Gb ( \mathcal{T}_n  ) $.

\medskip

We now consider the case: $(i,j) $ is in  $ \bar{\mathbb{G }}$. By Lemma \ref{tt11},
$O(R_i, R_j) =(0,0) $ and
$  O(R_j, R_i)= (0,1)    $. Thus, by \eqref{spiegare} follows that $N_{j,t} = V_{j,t}$ while
$$
N_{i,t} = V_{i,t} +  \, \sum_{s=0}^{\infty} \mathbf{1}_{ \{ Z_s \leq t -2\} }  \mathbf{1}_{  \{ X_{Z_s     }  \triangleright R_j \}  }     \mathbf{1}_{  \{ X_{Z_s  +1   }  \triangleright R_i  \}  }    .
$$
Now, defining $v_i $ and $v_j$ 
as in the proof of Lemma \ref{lemmastima}, $v_j=1$ while
$$
v_i \leq 1 - \frac{1}{N^k}  \left (1- \frac{  1 }{N^k}  \right )  <1.
$$
Therefore
$$
p_i   (   \{ T_{R_i}, T_{R_j}\}   ) = \frac{v_i}{ v_i +v_j}   <    \frac{v_j}{ v_i +v_j} =       p_j    (   \{ T_{R_i}, T_{R_j}\}   ) .
$$
Hence, $ (i,j)  $ belongs to $\Gb ( \mathcal{T}_n  ) $.
\end{proof}

We end the section with the following example

\begin{example}  \label{escostru2}
 We want to construct  $\mathbf{X}\in \mathcal{M}_1$ and three identically distributed hitting times such that 
\begin{equation}\label{T3T}
\mP(T_1 <T_2 ) > \frac{1}{2}, \qquad
 \mP(T_2 <T_3 ) > \frac{1}{2}, \qquad
\mP(T_3 <T_1 ) > \frac{1}{2}.
 \end{equation}
Thus we consider the digraph $ \bar \Gb $ and the generated patterns $R_1, R_2, R_3$ defined in Example \ref{escostru}. 
We take $\mathbf{X}= \mathbf{X}^{(4,4)}$ and $T_i = T_{R_i}$, for $i =1,2,3$. 
Now, by Theorem \ref{thgrfo}, the inequalities in \eqref{T3T} hold.
\end{example}

\section{A  Penney-type game}\label{comb}

The classical Penney's game concerns
 the occurrence of different strings in a sequence of independent random draws of letters. This kind of problem was studied and solved in \cite{chen1979} and \cite{GO1981} (see also \cite{DSS2} for a version of the  game with many players). In \cite{GO1981}, among other results, the Authors give the construction for the optimal reply or optimal string to every string chosen  by the first player. The game is always unfavorable for the player who chooses first.
The cause  of this behavior lies in the absence of transitivity for the stochastic  precedence order  (see e.g.
\cite{DSS1, Savage, trybula1969}).


To introduce our Penney-type game we need some notation. 
Let   $\mathcal{T}_n=\{ T_1, T_2 , \ldots , T_n \}$ be a collection of no-tie r.v.,
for $A \subset [n ]$ we write
$$T^{(A)}     = \min \{T_i : i \in A\}.$$
By the no-tie property, if the subsets $ A, B \subset [n ]$ are disjoint then $\mP (T^{(A)}      =T^{(B)}     ) =0$.

Let  $r_1, r_2 \in \N$,  $ \mathbf{X}  \in \mathcal{M}_1$ and  let  $   \mathcal{T}_n     $  be   a collection of $n $ identically distributed hitting times
on $\mathbf{X}$ with $n \geq r_1+r_2$. We define the stochastic zero-sum game  G$_{r_1,r_2 }(\mathbf{X} ,\mathcal{T}_n )  $ as follows:
\begin{itemize}
  \item[\emph{Step 1.}] Player \rom{1} chooses  a set  $A \subset [n]$ with $   |A|= r_1 $.
  \item[\emph{Step 2.}] Player \rom{2} chooses a set  $B \subset [n] \setminus A$   with $|   B|= r_2 $.
  \item[\emph{Step 3.}] Player \rom{1} chooses two nonempty sets  $A'  \subset A$  and $B'  \subset B$.
  \item[\emph{Step 4.}] If $T^{(A' )}   <  T^{(B')} $ then Player \rom{2} pays  $|B'| $ dollars  to Player \rom{1}, otherwise  Player \rom{1} pays $|A' | $ dollars       to Player \rom{2}.
\end{itemize}
The idea underlying this payoff is that, in the final  stage, each player pays one dollar for betting on any hitting time and the winner  takes all the stakes.
After the choice  of $A'$ and $B' $, the expected payoff of Player \rom{1} is given by
\begin{equation}\label{payoff}
 |B'| \, \cdot \, \mP ( T^{(A' )}   <  T^{(B')}    ) -
  |A'| \, \cdot \, \mP ( T^{(A' )}   >  T^{(B')}    ).
\end{equation}



Note that for given $ \mathbf{X} \in \mathcal{M}_1$ and for a collection of hitting times $\mathcal{T}_n$
the expected  payoff of the first player is a non-decreasing function
of $r_1$ and $ r_2 $, as long as $r_1+r_2 \leq n$.    Indeed, when $r_1' \geq r_1$ or $r_2'  \geq r_2$,  the
first player can mimic, for  G$_{r'_1, r'_2} (\mathbf{X} , \mathcal{T}_n)$,  the strategies  used in G$_{r_1, r_2} (\mathbf{X} , \mathcal{T}_n)$. Therefore his expected payoff is a monotone
increasing function in $r_1$ and $r_2$ when the two players adopt an optimal strategy.

It is quite easy to construct for given $ r_1$, $r_2$ and $n \geq r_1+r_2$  games of this kind that are fair or favorable to Player \rom{1}. We will come back to this point in this section. However,   we also determine a threshold $\calS (r_1, r_2)$ through a graph
characterization that will be used in the following result (see formula \eqref{def_grafica} below). 
\begin{theorem}\label{qaz} For  any $r_1, r_2 \in \N$,  
  there exist $  \mathbf{X} \in \mathcal{M}_1     $    and a family of identically distributed hitting times $        \mathcal{T}_n $ on
$  \mathbf{X}    $ such that the game G$_{r_1, r_2} (\mathbf{X}, \mathcal{T}_n )$
is favorable to Player \emph{\rom{2}}  if and only if
$n \geq \calS (r_1, r_2)$.
\end{theorem}
This result is also related with  "voting paradoxes", see e.g.
  \cite{AlonPara} and \cite{Dominating2006}.

\subsection{Existence of $(r_1, r_2)$-directional graphs}

We start this subsection with some definitions.
For a digraph $\Gb = ([n], \vec{E})$ and for disjoint  $A , B \subset [n]$ we write $A \to B $ if for any $i \in A $ and $j \in B$
one has that $(i,j) \in \vec{E}$.
For $r_1, r_2 \in \N$, we say   that a
digraph $\Gb=([n], \vec E)$ is $( r_1, r_2 )$-\emph{directional} if for any $A \subset  [n]$ with $|A|=r_1$ there exists $B \subset [n] \setminus A $ with $ |B| = r_2 $ such that $A \rightarrow B$ (see \cite{Erdos1963} and \cite{Dominating2006} for similar definitions). 
For any $r_1, r_2 \in \N $, let us define
\begin{equation}\label{def_grafica}
\mathcal{S} (r_1, r_2) : = \inf \left \{ k \geq r_1 + r_2: \text{there exists a $(r_1, r_2)$-directional tournament }
 ([k], \vec E)     \right \}   .
\end{equation}
In \cite{Erdos1963} Erd\H{o}s   analyses  a problem that correspond to the
existence of $(r_1, 1)$-directional graphs  (see also \cite{AlonSpencer}).    The probabilistic method developed there can be easily adapted in our case.

\begin{theorem}\label{directional}
For any $r_1, r_2 \in \N$ and  any $n \geq \mathcal{S} (r_1, r_2)$  there exists a $(r_1, r_2)$-directional  tournament $\mathbb{T} = ([n], \vec E)$. Moreover
\begin{equation} \label{boundE}
\mathcal{S} (r_1, r_2) \leq \inf \left  \{   n \geq r_1+r_2:  \binom{n}{r_1}  \left  (1- \frac{1}{2^{r_1 r_2}}  \right )^{  \left  \lfloor   \frac{n - r_1}{r_2}  \right   \rfloor         }    <1
    \right    \} < \infty .
\end{equation}
\end{theorem}
\begin{proof}
For $r_1, r_2 \in  \N$, we first assume that for a specific  $n_0 \in \N $ there exists
a $(r_1, r_2)$-directional tournament $    \mathbb{T}_{n_0} =([n_0] , \vec E)$. Then, we prove that for any $n > n_0 $
 there exists
a $(r_1, r_2)$-directional tournament $ \mathbb{T}_n =([n], \vec E_n)$. The proof is by induction.

Suppose that for $n-1 \geq n_0 $ there is  a   $(r_1, r_2)$-directional      tournament
$    \mathbb{T}_{n-1} =([n-1], \vec E_{n-1})$, then  we will construct a  tournament $    \mathbb{T}_{n} =([n], \vec E_{n})$ that is
$(r_1, r_2)$-directional.

For any distinct $i,j \in [n-1]$ let $(i,j)$ be in $\vec E_n$ if and only if
$(i,j) \in \vec E_{n-1}$. Moreover, for any  $i \in [n-1]$, we impose that $(n, i) $ belongs to $\vec E_n $.
It is clear that  if $ \mathbb{T}_{n-1}$ is a $(r_1, r_2)$-directional tournament then  also $ \mathbb{T}_n$ is an $(r_1, r_2)$-directional tournament. Indeed, if $A \subset [n-1]$ with $|A| =r_1$ then one can select  $B  \subset [n-1] $ with $|B| =r_2$ and $A \rightarrow B$,  as in  $ \mathbb{T}_{n-1}$. On the other hand if we consider an $A \subset [n]$ such that
$n \in A$ and $|A| = r_1$ then one can take  $B \subset [n-1]$ such that
$(A\setminus \{n\} )\rightarrow B $  and $|B| =r_2$. In any case the relation $(A\setminus \{n\} )\rightarrow B $ implies $A \rightarrow B $ because $(n, i) \in \vec E_n$ for any  $i \in [n-1]$.

Now, we  prove formula \eqref{boundE}, by the probabilistic method (see e.g.  \cite{AlonSpencer}).
 For this purpose, we will construct  a random tournament, denoted by $ \mathbb{T} (n )= ([n] , \vec{E} (n ))$, and  we will show that it is $(r_1, r_2)$-directional  with positive probability. 

For two distinct vertices $u,v$, either  $(u,v) \in  \vec{E} (n)$
or $ (v, u) \in \vec{E} (n) $;  both these events occur with probability $1/2$. Moreover all the  events involving   distinct edges  are assumed independent.

 For  given $r_1, r_2 \in \N$ let  $\widetilde V \subset [ n ]$ with $|  \widetilde        V|=r_1 $,
we define the event
$$
A_{\widetilde V} := \{ \exists V' \subset  [n ] \setminus \widetilde V :  \widetilde  V \to V', \text{ with } |V'| = r_2 \}.
$$
Now, for a given $\widetilde V $ having cardinality $ r_1$,  let us choose  a family of sets of vertices
$$
\left (V_i : |V_i|= r_2, V_i \subset [n]\setminus \widetilde V, \,\,\, i = 1, \ldots, \left  \lfloor   \frac{n - r_1}{r_2}  \right   \rfloor    \right ) ,
$$
with  $V_i \cap V_j = \emptyset $, for $i \neq j$.

By independence of the random directions involving different  edges one has
$$
\mP(A_{\widetilde V}^c) \leq  \mP \left ( \bigcap_{i=1}^{ \left  \lfloor   \frac{n - r_1}{r_2}  \right   \rfloor }  \{ \widetilde V \to V_i \}^c  \right  ) \leq   \left  (1- \frac{1}{2^{r_1 r_2}}  \right )^{  \left  \lfloor   \frac{n - r_1}{r_2}  \right   \rfloor         }          ,
$$
for any  $\widetilde V \subset [n] $ with $ | \widetilde   V|=r_1$.

\noindent
 By subadditivity of the probability measure one has
 \begin{equation}\label{cuscus}
 \mP \left ( \bigcap_{  \widetilde       V \subset [n ] : |    \widetilde      V|=r_1} A_{  \widetilde   V}  \right ) = 1- \mP \left ( \bigcup_{  \widetilde   V \subset [n ] : |  \widetilde   V|=r_1} A^c_{  \widetilde  V} \right ) \geq 1 -   \binom{n }{r_1}    \left  (1- \frac{1}{2^{r_1 r_2}}  \right )^{  \left  \lfloor   \frac{n - r_1}{r_2}  \right   \rfloor         }      .
 \end{equation}
 For any $r_1, r_2\in \N$,
\begin{equation}\label{limite0}
\lim_{n \to \infty}  \binom{n}{r_1}   \left  (1- \frac{1}{2^{r_1 r_2}}  \right )^{  \left  \lfloor   \frac{n - r_1}{r_2}  \right   \rfloor         }  =0.
\end{equation}
Formulas  \eqref{cuscus} and \eqref{limite0}
imply   \eqref{boundE}.
\end{proof}

Observe that, for a given $n \in \N$, if $\Gb= ([n], \vec E)$ is a $ (r_1,r_2)$-directional digraph and $ \Gb'= ([n], \vec E')$ with $\vec E \subset \vec E'$ then also $\Gb'$ is a $ (r_1,r_2)$-directional digraph.
This easy observation and Theorem \ref{directional} imply the following result.
\begin{corollary} \label{coro}
For $ r_1, r_2 \in \N$,  all the $(r_1, r_2)$-directional digraphs have a number of vertices larger than or equal to     $  \calS  (r_1, r_2) $.
\end{corollary}

\subsection{Favourable, fair and unfavorable games} \label{finale}
In the following, we take the point of view of the second player so we declare  \emph{favorable} (resp. \emph{fair}
and \emph{unfavorable}) the game if the expected value of the payoff of Player \rom{2}
is positive (resp. null and negative), when both players adopt optimal strategies.

We first construct, for any $n \geq r_1+r_2 $ a fair game G$_{r_1, r_2}( \mathbf{X} , \mathcal{T}_n)$.
 Let $\mathbf{X} =(X_m : m \in \N_0)$ be a sequence of i.i.d. random variables taking value on  $[n]$, with $X_0 $ uniformly  distributed  on $[n]$,
hence $\mathbf{X}\in \mathcal{M}_1$.  Let
$\mathcal{T}_n =\{ T_1, \ldots , T_n \}$ be the collection of identically distribute hitting times, where
$T_i= \inf \{ m \in \N_0 :X_m = i \}, \text{ for } i \in [n] .$
For any strategy of the players the game has a null expected payoff, therefore the game is trivially fair.

Now we construct an unfavorable game for any $n \geq r_1+r_2$.   Let us consider the Markov chain $\mathbf{X}^{(n+1, n+1)}$, a tournament
$\mathbb{T} = ([n], \vec E)$ with the property that $(i, n) \in \vec E$ for any $i \in [n-1]$, the patterns $R_1, \ldots , R_n $ generated by
$\mathbb{T}$ and the hitting times $T_{R_1}, \ldots , T_{R_n}$.  Player \rom{1} takes
$A$ with $n \in A  $ then, for any chosen set $B$ by Player \rom{2},  Player \rom{1} selects $A' = \{n \}$ and $B' \subset B $ with $|B'| =1$. By construction $B \to \{n \}$,  therefore Player \rom{1} has guaranteed  a positive expected payoff despite of the fact that this strategy could be suboptimal.

In order to find for which $n $ there exist  favorable games we need some more discussions and definitions. The following Definition \ref{small} and \ref{2-det} are similar
to others given in \cite{DMS19} but they are used there for different applications and purposes.

\begin{definition} \label{small}
Let us consider two finite sets of random variables $\mathcal{S}_A =\{ S_i : i \in  A \} $ and $\mathcal{S}_B =\{ S_i : i \in B \}$, such that $\mathcal{S}_A \cup \mathcal{S}_B$ has the no-tie property. We say that
$\mathcal{S}_A$ is \emph{small} with respect to $\mathcal{S}_B$ iff
\begin{equation}\label{ssmm}
\frac{1}{|A|}\sum_{i \in A} p_i (   \mathcal{S}_A \cup \mathcal{S}_B  ) >\frac{1}{|B|}\sum_{i \in B} p_i (   \mathcal{S}_A \cup \mathcal{S}_B  ).
\end{equation}
\end{definition}

\medskip

First we present an example showing that the collective behaviour cannot be deduced by pair relations.

\begin{example}
Let $S_1 = \frac{49}{100}$ and let $S_2, S_3 $ be independent r.v. with uniform law on $[0,1]$. The collection $\{S_1, S_2,S_3 \}$ has the no-tie property.
The r.v.
 $S_1$ is small with respect to $S_i $, for $i=2,3$, because
 $\mP (S_1< S_2) =   \mP (S_1< S_3) =     \frac{51}{100}$. But
 $$
 \frac{1}{2} \sum_{i=2}^3  p_i     (\{ S_1, S_2, S_3 \})  =  \frac{1}{2}  -  \frac{1}{2}  \left (\frac{51}{100}  \right )^2  > \left (\frac{51}{100}  \right )^2 =
 p_1    (\{ S_1, S_2, S_3 \})   .
 $$
Therefore $\{ S_2, S_3\}$ is small with respect
to $S_1$. This example shows that the analysis of the
smallness property cannot be reduced to the study of
pair relations. In fact, these  can be completely reversed
when we move on to consider collections of random variables. A complete theory that studies all the possible ranking in a set of  random variables is in
\cite{DS20} and \cite{Saari}.
\end{example}

In order to avoid some difficulties in the construction of favorable games we define some special systems of random variables.

  \begin{definition} \label{2-det}
 Let $\mathcal{S}_n = \{S_1, \ldots , S_n \}  $ be a collection of random variables  and let $\Gb( \mathcal{S}_n ) =  ([n], \vec{E})$ be the associated ranking graph. We say that
 $\mathcal{S}_n$ is  $2$-\emph{determined} if  for any two disjoint $A, B \subset [n ]$ such that $A \to B $ then
 $\calS_B =\{ S_i : i \in B \} $ is small with respect to $\calS_A =\{ S_i : i \in A\} $.
  \end{definition}

\begin{theorem}\label{THs}
Let $n \geq 2$ and $ \bar \Gb =([n], \vec E)$ be given. Let  $R_1, \ldots , R_n $ be patterns
in $\mathcal{P}_{n+1,n+1}$ generated by
$\bar \Gb =([n], \vec E)$.
For $N \geq n+1$, let us consider the 1-dependent uniform
chain  $\mathbf{X}^{(N,n+1)}$ and the identically  distributed  hitting times
$\mathcal{T}_n = \{T_{R_1} , \ldots ,  T_{R_n}\} $.
 Then 
 $\mathcal{T}_n$ is $2$-determined. 
\end{theorem}
\begin{proof} 
For any $C \subset [n]$ let $\mathcal{T}_{C } :=\{   T_i: i \in C \}$. For any disjoint $A, B \subset [n]$ with $A \to B $,
we need to show
\begin{equation}\label{ssmm2}
\frac{1}{|A|}\sum_{i \in A} p_i (   \mathcal{T}_A \cup \mathcal{T}_B  ) <\frac{1}{|B|}\sum_{i \in B} p_i (   \mathcal{T}_A \cup \mathcal{T}_B  ) .
\end{equation}
Inequality \eqref{ssmm2} holds true if and only if
$ \frac{1}{|B|} \sum_{i \in B} v_i-     \frac{1}{|A|} \sum_{i \in A} v_i >0 $, , for disjoint sets $A, B $. 

\noindent
For any $i \in A$, by  Lemma \ref{lemmastima}, one has
$$
v_i \leq 1-\frac{1}{N^{n+1}} \left (  1- \frac{|A|+|B| -1}{N^{n+1}} \right )   \frac{|B|}{N^{n+1}}  .
$$
 Hence,
\begin{equation}\label{Ape}
\frac{1}{|A|} \sum_{i \in A} v_i \leq 1-\frac{|B|}{N^{n+1}} + \frac{   (   |A| + |B| -1 )|B|  }{N^{2n+2}} .
\end{equation}
Analogously, by the first inequality in \eqref{vistima} and the hypothesis that $A \to B$, one has
\begin{equation}\label{Bpe}
\frac{1}{|B|}  \sum_{i \in B } v_i \geq 1 - \frac{1}{N^{n+1}} \sum_{i \in B}\,\, \sum_{j \in B \setminus \{ i\}  : (i,j) \in \vec E   } 1   .
\end{equation}
Hence,
\begin{equation}\label{Bpe2}
\frac{1}{|B|}  \sum_{i \in B } v_i \geq 1 - \frac{|B|  -1 }{2N^{n+1}}
\end{equation}
Therefore  one  obtains 
\begin{equation}\label{NuovA}
\frac{1}{|B|} \sum_{i \in B} v_i-     \frac{1}{|A|} \sum_{i \in A} v_i >
\frac{|B|}{2N^{n+1}}  - \frac{   (   |A| + |B| -1 )|B| }{N^{2n+2}}     > |B| \cdot \left [ \frac{1}{2N^{n+1}}    -
 \frac{1}{N^{2n+1}} \right ]  ,
\end{equation}
where the last inequality is a consequence of  $N > |A| + |B|  -1$.
Thus, for any $N \geq  n+1 \geq 3$ and for any choice of non-empty disjoint sets $A, B \subset  [n]$ the l.h.s. in \eqref{NuovA} results larger than zero.
\end{proof}

We give the key result that will allow the construction of  favorable games.
\begin{theorem}\label{THgame}
Let   $r_1, r_2 \in \N $, $n \geq r_1+r_2$,  let $\mathbf X \in \mathcal{M}_1$  and   $\mathcal{T}_n = \{   T_1, \ldots , T_n \}$
be a collection of identically distributed hitting times.
\begin{itemize}
\item[i.] The game G$_{r_1, r_2} (\mathbf{X}  , \mathcal{T}_n )$ is favorable
$\Rightarrow $   $ \Gb (\mathcal{T}_n) $ is  $(r_1, r_2)$-directional.
\end{itemize}
 Moreover,  let us suppose that $ \mathcal{T}_n $ is $2$-determined then
 \begin{itemize}
 \item[ii.] The game G$_{r_1, r_2} (\mathbf{X}  , \mathcal{T}_n )$ is favorable
 $\iff $   $ \Gb (\mathcal{T}_n) $ is  $(r_1, r_2)$-directional.
 \end{itemize}
\end{theorem}
\begin{proof}
Item i. The proof is by contradiction.  Suppose that $\Gb (\mathcal{T}_n) = ( [n] , \vec E ( \mathcal{T}_n) )$ is not $(r_1, r_2)$-directional.
By hypothesis,
  Player \rom{1} can select $A$, with cardinality $r_1$, such that,  for any
 $B$ with cardinality $r_2$,    $A \not \to B$ in $ \Gb (\mathcal{T}_n) $. Let us consider such  $A$ and $B$. Now, Player \rom{1}  selects  $i \in A$ and $j \in B $ such that  $(i,j) \notin \vec E ( \mathcal{T}_n)$. Hence, $\mathbb{P} (T_j < T_i) \leq  \frac{1}{2}$. Then, by choosing $A' = \{  i  \} $ and
$B' = \{ j\} $ Player \rom{1}  has guaranteed that the game is either fair or in his favour. Obviously,   this strategy could be suboptimal for Player \rom{1}.  
In any case, the game G$_{r_1, r_2} (\mathbf{X}  , \mathcal{T}_n )$ is not favorable for Player \rom{2}.

Item ii. 
For any  set $A$ having cardinality $r_1$, Player \rom{2}  selects  $B$, with cardinality $r_2$,  such that $A \to B$. Let us consider such a $B$. From the fact that the system is $2$-determined for any choice of no empty $A' \subset A $ and  $B' \subset B $
one obtains
$$
\frac{1}{|A'|} \sum_{i \in A'} p_i (\mathcal{T}_{A' \cup B'}) <
\frac{1}{|B'|} \sum_{i \in B'} p_i (\mathcal{T}_{A' \cup B'}) .
$$
By \eqref{payoff},  the expected payoff of Player \rom{2} is positive,
for all  $A' \subset A $, $B' \subset B$.
\end{proof}

\medskip

Now we prove Theorem \ref{qaz}  presented in the beginning of this section.
For given $r_1, r_2$ if $n < \calS (r_1, r_2)$, by  Corollary \ref{coro} any graph is  not $(r_1, r_2)$-directional;
then by Theorem \ref{THgame},   the considered  game is fair or  unfavorable.

If $n \geq  \calS (r_1, r_2) $, by Theorem \ref{directional}, one take 
a digraph $\bar \Gb  =([n], \vec E )$  that is $(r_1, r_2)$-directional.  Let us take 
a chain $\mathbf{X}^{(N, n+1)} \in \mathcal{M}_1$ with $N \geq n+1$.  Then  the  patterns $R_1, \ldots , R_n \in \mathcal{P}_{n+1, n+1} $ generated by  $\bar \Gb$ are considered.
By  Theorem \ref{thgrfo},  we know that $ \Gb  (\{ T_{R_1}, \ldots  ,  T_{R_n}  \}) = \bar \Gb $, moreover  $T_{R_1}, \ldots , T_{R_n}$ are identically distributed 
hitting times and,  by Theorem \ref{THs}, these hitting times  are 
 $2$-determined. 
Finally, by the item ii of Theorem \ref{THgame}, the game G$_{r_1, r_2} (\mathbf{X}^{(N, n+1)}, \{  T_{R_1}, \ldots , T_{R_n}  \}  )$ is favorable. This proves Theorem \ref{qaz}.

\bigskip

\noindent
{\bf Acknowledgements.}   
I would like to thank the anonymous Reviewers for several helpful comments and suggestions that have improved the text.

\bibliographystyle{abbrv}



\end{document}